\RequirePackage{fix-cm}
\documentclass[a4paper,times]{llncs}
\usepackage{llncsdoc}
\usepackage[T1]{fontenc}
\usepackage[utf8]{inputenc}
\usepackage{graphicx}
\usepackage{hyperref}
\usepackage{amssymb}
\usepackage{amsmath,bm}
\usepackage{lmodern} 
\usepackage{textcomp} 
\usepackage{sectsty}
\usepackage{fancybox}
\usepackage{listings}
\usepackage{charter}
\usepackage[titletoc]{appendix}
\usepackage{epsfig}
\usepackage{tabularx}
\usepackage{authblk}
\usepackage{enumitem}
\usepackage[ruled,vlined]{algorithm2e}
\usepackage{latexsym}
\usepackage{alltt}
\usepackage{setspace}
\usepackage{datetime}
\newdateformat{mydate}{\monthname[\THEMONTH], \THEYEAR}
\def\qed{\hfill$\Box$}

\newtheorem{construction}{Construction}

\newtheorem{defn}{\textbf{Definition}}

\newenvironment{statement}[1][Statement]{\noindent {\it #1}.}{}

\usepackage[showframe=false, left=3.75cm, top=3cm, right=3cm, bottom=3cm ]{geometry}
\newcommand{\Mod}[1]{\ (\mathrm{mod}\ #1)}

\makeatletter
\renewcommand{\@Opargbegintheorem}[4]{%
	#4\trivlist\item[\hskip\labelsep{#3#2\@thmcounterend}]}
\makeatother

\begin{document}

	\renewcommand\Authands{ and }
	\title{System of unbiased representatives for a collection of bicolorings}
%
	
	\author{Niranjan Balachandran$^1$\and
	Rogers Mathew$^2$\and
	Tapas Kumar Mishra$^2$ \and \\
	Sudebkumar Prasant Pal$^2$}
	
	\institute {Department of Mathematics, Indian Institute of Technology, Bombay 400076, India \email{niranj@iitb.ac.in}\and
	Department of Computer Science and Engineering, Indian Institute of Technology, Kharagpur 721302, India \\
	\email{\{rogers,tkmishra,spp\}@cse.iitkgp.ernet.in}}

	\maketitle
	\begin{abstract}
		
		%
		Let $\mathcal{B}$ denote a set of bicolorings of $[n]$,
		where each bicoloring is a mapping of the points in $[n]$ to $\{-1,+1\}$.
		For each $B \in \mathcal{B}$, let $Y_B=(B(1),\ldots,B(n))$.
		For each $A \subseteq [n]$, let $X_A \in \{0,1\}^n$ denote the incidence vector of $A$.
		A non-empty set $A$ is said to be an `unbiased representative' for a bicoloring $B \in \mathcal{B}$ if
		$\left\langle X_A,Y_B\right\rangle =0$.
		Given a set $\mathcal{B}$ of bicolorings, we study the minimum cardinality of a family $\mathcal{A}$ consisting of subsets of $[n]$ such that every bicoloring in $\mathcal{B}$ has an unbiased representative in $\mathcal{A}$.
	\end{abstract}
	%
	
	
	
	\section{Introduction}

	Let $\mathcal{B}$ denote a set of bicolorings of $[n]=\{1,\ldots,n\}$,
	where each bicoloring $B \in \mathcal{B}$ maps each point $x \in [n]$ to either -1 or +1.
	Let $Y_B$ denote the $n$-dimensional vector representing the bicoloring $B$, i.e. $Y_B=(B(1),\ldots,B(n))$.
	A non-empty set $A \subseteq [n]$ is said to be an \emph{unbiased representative} for a bicoloring $B \in \mathcal{B}$ 
	if $\left\langle X_A,Y_B\right\rangle =0$, where $X_A$ denotes the 0-1 $n$-dimensional incidence vector corresponding to $A$.
	We call a family $\mathcal{A}$ of subsets of $[n]$ a  \emph{system of unbiased representatives} (or `SUR') for 
	$\mathcal{B}$  if for every bicoloring $B  \in \mathcal{B}$, there exists at least one set $A \in \mathcal{A}$
	such that $\left\langle X_A,Y_B\right\rangle =0$.
	Note that the two monochromatic bicolorings can never have any unbiased representatives -
	we call these bicolorings `trivial'.
	Let $\gamma(\mathcal{B})$ denote the minimum cardinality of a system of unbiased representatives for $\mathcal{B}$.
	We define the maximum of $\gamma(\mathcal{B})$ over all possible families $\mathcal{B}$ of non-trivial bicolorings of $[n]$
	as $\gamma(n)$.  Note that no singleton set of $[n]$ is a member of any optimal system of unbiased representatives.
	
	Unbiased representatives are useful in testing products such as drugs over a large population
	where the effectiveness (or side-effect) of a new drug is studied in correlation with a large set 
	of patient attributes such as body weight, height, age, etc. Complementary extremes in the 
	attributes, such as being obese or underweight, tall or short, and young or old, are 
	relevant is such correlation studies.
	Such studies require patients with complementary ranges of values of a certain attribute to be present in
	equal (or roughly equal) numbers in the representative group for that attribute -- such a group may 
	be deemed to be an unbiased representative for the attribute.
	%
	%
	%
	However, selecting a separate sample of individuals for each attribute having equal representation of the complementary traits is practically impossible. 
	So, one needs to select a family $\mathcal{A}$ of samples of individuals such that 
	for any attribute $B$,	there exists a sample $A\in \mathcal{A}$ which has an equal representation of individuals from 
	the complementary traits of $B$.
	It is in the best interest to choose a family $\mathcal{A}$ of such groups of representatives of the smallest possible cardinality.
	It is not hard to see the direct mapping of this problem to the problem addressed in this paper.
	In a generic setting, SURs are useful in various applications 	where a collection of items (like individual patients) have many attributes (like weight, height and age), where 
	the objective is to form a small collection of subsets of items with almost equal representation of 
	opposite or complementary traits for each attribute.

	\subsection{Definitions and notations}
	
	We use `SUR' to denote the phrase `system of unbiased representatives'.
	For integers $n$ and $p$, let $[n]$ denote the set $\{1,\ldots,n\}$, and
	$[n\pm p]$ denote the set $\{n-p,n-p+1, \ldots,n+p\}$.
	A bicoloring $B$ of $[n]$ is called a $k$-\emph{bicoloring} if the number of +1's in $B$ is exactly $k$.
	For a bicoloring $B:[n] \rightarrow \{-1,1\}$, we use $B(+1)$ (respectively, $B(-1)$) to denote the set of points receiving color
	+1 (respectively, -1) under $B$.
	We use $Y_B$ ($X_A$) to denote the $n$-dimensional $\pm 1$ vector (respectively, 0-1 vector) representing the bicoloring $B$
	(respectively, $A \subseteq [n]$), i.e. $Y_B=(B(1),\ldots,B(n))$.
	Note that $\langle Y_B,X_A \rangle=0$ for some $A \in \binom{[n]}{r}$ implies that that $r$ is even.
	Throughout the rest of the paper, we consider only the non-trivial bicolorings and assume that every set in a SUR is of even cardinality.
	Let $\gamma(\mathcal{B},k,r)$ (respectively, $\gamma(\mathcal{B},[k_1,k_2],[r_1,r_2])$) be the minimum cardinality of a SUR  $\mathcal{A}$ for $\mathcal{B}$, where
	(i) each $B \in \mathcal{B}$ is a bicoloring of $[n]$ consisting of exactly $k$ +1's
	(respectively, at least $k_1$ and at most $k_2$ +1's), and,
	(ii) each $A \in \mathcal{A}$ is an $r$-sized 
	(respectively, at least $r_1$-sized and at most $r_2$-sized) subset of $[n]$.
	We define $\gamma(n,k,r)$ ($\gamma(n,[k_1,k_2],[r_1,r_2])$) as follows.
	\[  \gamma(n,k,r) = \max_{\mathcal{B}} \gamma(\mathcal{B},k,r).\]
	\[  \gamma(n,[k_1,k_2],[r_1,r_2]) = \max_{\mathcal{B}} \gamma(\mathcal{B},[k_1,k_2],[r_1,r_2]).\]
	Since no singleton set of $[n]$ can be a member of any optimal system of unbiased representative and
	the monochromatic bicolorings, consisting of exactly zero (or $n$) +1's, are trivial,
	$\gamma(n,[1,n-1],[2,n])$ is the same as $\gamma(n)$.

	\subsection{Relation to existing works}
	
	Given a family $\mathcal{F}$ of subsets of $[n]$, finding another family $\mathcal{F}'$ with certain properties 
	in relation with $\mathcal{F}$ has been well investigated. 
	One of the most studied problem in this direction is the computation of \emph{separating families}(see \cite{kat1966}).
	Let $\mathcal{F}$ consist of pairs $\{i,j\}$, $i,j \in [n]$, $i \neq j$ and $\mathcal{S}$
	be another family of subsets on $[n]$.
	A subset $S$ separates a pair $\{i,j\}$ if $i \in S$ and $j \not\in S$ or vice versa.
	The family $\mathcal{S}$ is a separating family for $\mathcal{F}$ if every pair $\{i,j\} \in \mathcal{F}$ is separated by some  $S \in \mathcal{S}$ (see \cite{renyi1961,kat1966,weg1979,dick1969,spencer1970} for detailed results and related problems on separating families). Separating families have many applications like `Wasserman-type' blood tests of large populations,
	diagnosis and chemical analysis, locating defective items, etc (see \cite{kat1973}).
	An extension of the separating family problem is the `test cover' problem: ``Given a family $\mathcal{F}$ of subsets of $[n]$,
	finding a sub-collection   $\mathcal{T} \subseteq \mathcal{F}$ of minimum cardinality such that
	every pair of $[n]$ is separated by some $S \in \mathcal{T}$''.
	The test cover problem is studied in the context of drug testing, biology \cite{payne1980identification,willcox1972automatic,lapage1973identification} and pattern recognition \cite{devijver1982pattern}.
	For results and related notions, see 
	\cite{moret1985minimizing,halldorsson2001approximability,de2002branch,crowston2012parameterized,basavaraju2016partially}.
	In the above problems, any two sized set $F=\{i,j\}$ can be viewed as a partial bicoloring $\chi:[n] \rightarrow \{-1,0,1\}$ where 
	$\chi(i)=-1$, $\chi(j)=+1$, and $\chi(p)=0$ for any $p \in [n]\setminus \{i,j\}$ and a set $S$ covers $F$ if and only if
	$\left\langle X_S,Y_{\chi}\right \rangle \in \{-1,+1\}$.
	
	An affine hyperplane is a set of vectors $H(a,b)=\{x \in \mathbb{R}^n: \left< a,x\right>=b\}$,
	where $a \in \mathbb{R}^n$ is a nonzero vector, $b \in \mathbb{R}$. Covering the $\{0,1\}^n$ Hamming cube
	with the minimum number of affine hyperplanes has been well studied - a point $x \in \{0,1\}^n$ is said to be \emph{covered} by a hyperplane 	$H(a,b)$ if $\left< a,x\right>=b$ (see \cite{Alon1993,Linial2005,sax2013}).
	It is not hard to see that 
	any SUR for the $2^{n}-2$ non-trivial bicolorings is a covering 
	for all the points of the $\{-1,1\}^n$ Hamming cube, except $\{(-1,\ldots,-1),(1,\ldots,1)\}$, 
	by hyperplanes $H(a,b)$ satisfying (i) $a \in \{0,1\}^n$ and (ii) $b=0$.
	
	\subsection{Summary of results}
	
	The paper is divided into three logical sections.
	The first section (Section \ref{sec:1chap3}) focuses on obtaining $O(\log |\mathcal{B}|)$ upper bounds for SURs when 	
	(i) the collection $\mathcal{B}$ of bicolorings is unrestricted or has minor restrictions, and
	(ii) the sets in the SURs are  unrestricted or have minor restrictions.
	When $\mathcal{B}$ consists of all the $2^n-2$ non-monochromatic bicolorings, 
	it is not difficult to show that $\frac{n}{2} \leq \gamma(\mathcal{B},[1,n-1],[2,n]) \leq n-1$.
	Using a nice application of Combinatorial Nullstellensatz \cite{AlonNull1999}, we improve the above lower bound to
	$n-1$.
	
	\begin{theorem}
		Let $n$ be a positive integer and $k \in [n]$. Then,
		$\gamma(n,[1,n-k],[2,n]) = n-1$, where $1 \leq k \leq \lceil\frac{n}{2}\rceil$.
		\label{thm:unrestrictchap3}
	\end{theorem}
	
	We relate the problem of SUR to the hitting set problem, which in turn implies relations with `VC-dimension'
	provided $\epsilon n \leq |B(+1)| \leq (1-\epsilon)n$ for each	$B \in \mathcal{B}$.
	For such families $\mathcal{B}$, this relationship assists in establishing an $O(\log |\mathcal{B}|)$ upper bound
	for cardinalities of any optimal SUR.
	Under a similar restriction for each $B \in \mathcal{B}$, if it is mandatory that each set in the SUR is of cardinality exactly $r$, 
	the best upper bound obtained is large ($\Omega(\sqrt{r}\log |\mathcal{B}|)$).
	In order to establish an $O(\log |\mathcal{B}|)$ upper bound for size of an optimal SUR
	under this restriction, we introduce some error in the representations
	and we have the following theorem.
	
	\begin{theorem}
		Let $r' \in [r \pm \lceil\frac{r}{2}\rceil]$, where $r \geq 8$ is an integer.		
		Let $\mathcal{B}$ denote the set of all bicolorings $B \in \{-1,+1\}^n$,
		where $|B(+1)-B(-1)| \leq d$, for some $d \in \mathbb{N}$.
		Then, with high probability, one can construct a family $\mathcal{A}$ of cardinality at most $\ln |\mathcal{B}|$
		in $O(n|\mathcal{B}|\ln |\mathcal{B}|)$ time 
		consisting of $r'$-sized subsets such that for every $B \in \mathcal{B}$, there exists a set $A \in \mathcal{A}$ with
		$|\left\langle Y_B,X_A\right\rangle| \leq e\sqrt{r}+\frac{dr}{n}$.
		\label{thm:biasedchap3}
	\end{theorem}
	
	In the second part of the paper (Section \ref{sec:2chap3}), we study the SUR problem where each $B \in \mathcal{B}$ is restricted to have exactly $k$ +1's
	and each set in the SUR is required to be of cardinality exactly $r$, for some $r,k \in [n]$, $2 \leq r \leq 2k$.
	We relate the SUR problem under such restrictions to `covering' problems, that enables us to use a deterministic algorithm of
	Lov\'{a}sz \cite{LOVASZ1975383} and Stein \cite{STEIN1974391} to compute such a SUR in polynomial time. 
	In particular, for sufficiently large values of $n$, and $k \leq \log_4 \log_4 (n^{0.5-\epsilon})$, we use a  
	result of Alon et al. \cite[Corollary 1.3]{Alon2003}
	to establish the following asymptotically tight bound on $\gamma(n,k,2k)$.
	
	\begin{theorem}
		For sufficiently large values of $n$,
		\[\frac{\binom{n}{k}} {\binom{2k}{k}} \leq \gamma(n,k,2k) \leq \frac{\binom{n}{k}} {\binom{2k}{k}}(1+o(1)), \]
		provided $k \leq \log_4 \log_4 (n^{0.5-\epsilon})$, for any $0 < \epsilon < 0.5$.
		\label{thm:nkrtightchap3} 	
	\end{theorem}
	
	The problem of estimation of $\gamma(n,k,r)$ becomes interesting when $k=\frac{n}{2}$ - the reduction to coverings
	gives a lower and upper bound of $\max\left(\left\lceil\frac{n}{2r}\right\rceil, c_1\sqrt{\frac{r(n-r)}{n}} \right)$ and $O(n\sqrt{\frac{r(n-r)}{n}})$, respectively. For $r=f(n)$, where $f(n)$ is an increasing function in $n$, this establishes only sub-linear lower bounds for $\gamma(n,\frac{n}{2},r)$. 
	We use a vector space orthogonality argument combined with a theorem of Keevash and Long \cite{keevash2017frankl}
	to obtain a linear lower bound on $\gamma(n,k,r)$ under certain restrictions on $n$, $k$ and $r$.
	
	\begin{theorem}
		Let $r=2c$
		for any odd integer $c \in \{1,\ldots,\frac{n}{2}\}$. 
		Let $k$ be an even integer, where $\epsilon n < k < (1-\epsilon)n$ for some $0 < \epsilon < 0.5$.
		Then, $\gamma(n,k,r) \geq \delta n$, 	
		where $\delta=\delta(\epsilon)$ is some real positive constant.
		\label{thm:keevashchap3}
	\end{theorem}
	
	Combined with an upper bound construction given in Lemma \ref{lemma:nn2upchap3},
	this establishes an asymptotically tight bound for $\gamma(n,\frac{n}{2},\frac{n}{2})$, when $\frac{n}{2} \equiv 2 \Mod{4}$.
	
	In the third part of the paper (Section \ref{sec:3chap3}), 
	we obtain the following inapproximability result for computing optimal SURs by using a result  of Dinur and Steurer \cite{dinur2014}
	on the inapproximability of the hitting set problem.

	\begin{theorem}
		Let $r \leq  (1-\Omega(1))\frac{\ln n}{2.34}$, where $n \geq 4$ is an integer.
		Then, no deterministic polynomial time algorithm can 
		approximate the system of unbiased representative problem for a family of bicolorings on $[n]$
		to within a factor $(1-\Omega(1))\frac{\ln n}{2.34r}$ of the optimal 
		when each set chosen in the representative family is required to have its cardinality at most $r$, unless P=NP.
		\label{thm:inappchap3}
	\end{theorem}
	
	\section{When cardinalities of sets in the `SUR' are unrestricted or semi-restricted}
	\label{sec:1chap3}
	\subsection{Bounds on $\gamma(n,[k,n-1],[2,n])$}
	\label{subsec:1.1chap3}
	
	Recall that $\gamma(n)=\max_{\mathcal{B}} \gamma(\mathcal{B})$,
	where $\gamma(\mathcal{B})$ is the cardinality of an optimal
	system of unbiased representative for $\mathcal{B}$.
	Observe that
	$\gamma(\mathcal{B}_1) \leq \gamma(\mathcal{B}_2)$ when
	$\mathcal{B}_1 \subseteq \mathcal{B}_2$.
	So, to establish bounds on $\gamma(n)$, it suffices to consider the set of all the $2^n-2$ non-monochromatic bicolorings as $\mathcal{B}$ and establish bounds on $\gamma(\mathcal{B})$.
	We have the following proposition.
	
\begin{proposition}
	Let $n$ be an integer and $k \in [n]$.
	\begin{enumerate}[label=(\roman*)]
\item \label{p:eq1} $\gamma(n,[k,n-1],[2,n]) =\gamma(n,[1, n-k],[2,n])$.
\item \label{p:eq2} $\gamma(n,[1, n-k],[2,n]) = \gamma(n,[1, \lfloor\frac{n}{2}\rfloor],[2,n])$, for any $1 \leq k \leq \lceil\frac{n}{2}\rceil$.
\item \label{p:eq3} $\gamma(n,[1,n-k],[2,n]) \leq n-1$, for $1 \leq k \leq n$.
\item \label{p:eq4} $\frac{n}{2} \leq \gamma(n,1,[2,n]) \leq \gamma(n,[1,n-k],[2,n])$, for $1 \leq k \leq n-1$.
	\end{enumerate}
\label{prop:chap3easychap3}
\end{proposition}
	
\begin{proof}
	
\ref{p:eq1}  For any $k$-bicoloring $B$, any unbiased representative $A$ for $B$ is also an unbiased representative
for the bicoloring $B'$, where $B'(+1)=B(-1)$ and $B'(-1)=B(+1)$.\\
\ref{p:eq2} The proof follows from the proof of Statement \ref{p:eq1} in Proposition \ref{prop:chap3easychap3}.\\
\ref{p:eq3} Let $\mathcal{B}$ denote the set of all the $2^n-2$ non-monochromatic bicolorings.
It is not hard to see that $\mathcal{A}=\{\{1,2\},\{1,3\},\ldots,\{1,n\}\}$
is a SUR of cardinality $n-1$ for $\mathcal{B}$. \\
\ref{p:eq4} 	Let $\mathcal{B}=\{B| |B(+1)|=1 \}$. So, $|\mathcal{B}|=n$.
For any $B \in \mathcal{B}$, if for any $A \subseteq [n]$,
$\left\langle Y_B,X_A \right\rangle=0$, then $|A|=2$.
Moreover, for any $A \in \binom{[n]}{2}$, exactly two $B \in \mathcal{B}$ has $\left\langle Y_B,X_A \right\rangle=0$.
So, we need at least $\frac{n}{2}$ two sized sets to form a SUR for $\mathcal{B}$. The second inequality follows from the containment.
\qed
\end{proof}	

In the construction leading to the proof of Statement \ref{p:eq3} in Proposition \ref{prop:chap3easychap3},
only two-sized sets are used as unbiased representatives. We have the following slightly non-trivial construction assuming $n=2^p$, for some integer $p$, giving similar bounds. 
Let $\mathcal{A}_2=\{\{1,2\},\{3,4\},\ldots, \{n-1,n\}\}$ : a partition of $[n]$ into two-sized sets.
Let $\mathcal{A}_4=\{\{1,2,3,4\}, \{5,6,7,8\},\ldots, \{n-3,n-2,n-1,n\}\}$ : a partition of $[n]$ into four-sized sets taken in that order. Similarly, repeating the construction for $p-2$ more steps, we obtain a sequence of partitions of $[n]$, 
$\mathcal{A}_2,\mathcal{A}_4,\ldots, \mathcal{A}_n$, where $\mathcal{A}_i$ is a partition of $[n]$ into $i$-sized $\frac{n}{i}$ parts, 
i.e., $\mathcal{A}_{i}=\{\{1,\ldots,i\}, \{i+1,\ldots,2i\}, \ldots, \{n-i+1,\ldots,n\} \}$.
Let $\mathcal{A} = \mathcal{A}_2 \cup \mathcal{A}_4 \cup \cdots \cup \mathcal{A}_n$. It follows that $|\mathcal{A}|=2^{p-1}+2^{p-2}+\ldots+1=2^{p}-1=n-1$.
To see that this is indeed a SUR for the set of all the $2^n-2$ non-monochromatic bicolorings,
let $B \in \{-1,1\}^n$ denote any non-trivial bicoloring of $[n]$.
Without loss of generality, assume that $|B(+1)| \leq |B(-1)|$.
Let $i$ ($2 \leq i \leq n$)  be the minimum index such that there exists an  $A \in \mathcal{A}_{i}$ with $A \setminus B(+1) \neq \phi$
and $A \cap B(+1) \neq \phi$. 
From construction of $\mathcal{A}_{i}$ and assumption on $i$, it follows that
there exists consecutive parts $A_1,A_2 \in \mathcal{A}_{\frac{i}{2}}$ with $A_1 \subseteq B(+1)$, $A_2 \cap B(+1) = \phi$, and $A = A_1 \cup A_2$.
So, it follows that $A$ is an unbiased representative for $B$.

To establish a tight lower bound on $\gamma(n,[1,\lceil\frac{n}{2}\rceil],[2,n])$ ($\gamma(n,[1,n-1],[2,n])$), we 
need the following lemma.

\begin{lemma}
	Let $F \in \mathbb{F}(x_1,\ldots,x_n)$ be a polynomial and
	$S_1,\ldots,S_n$ be non-empty subsets of $\mathbb{F}$,
	for some field $\mathbb{F}$. If $F$ vanishes on all but one point
	$(s_1,\ldots,s_n) \in S_1 \times \cdots \times S_n \subseteq \mathbb{F}^n$,
	then deg($F$) $\geq \sum_{i=1}^{n}(|S_i|-1)$.
	\label{lemma:nullchap3}
\end{lemma}

\begin{proof}
	For the sake of contradiction, assume that
	deg($F$) $< \sum_{i=1}^{n}(|S_i|-1)$.
	Consider the polynomials.
	\begin{align*}
	H(x_i)= \prod_{s \in S_i \setminus\{s_i\}}(x_i-s). \\
	G(x_1,\ldots,x_n)= \prod_{i=1}^{n}H(x_i).
	\end{align*}
	Note that deg($G$) is $\sum_{i=1}^{n}(|S_i|-1)$.
	Let $F(s_1,\ldots,s_n)=c_1$ and $G(s_1,\ldots,s_n)=c_2$.
	Then, the polynomial $c_2F-c_1G$ vanishes on all points of $S_1 \times \cdots \times S_n$.
	However, $c_2F-c_1G$ has degree $\sum_{i=1}^{n}(|S_i|-1)$: the monomial $x_1^{|S_1|-1}\cdots x_n^{|S_n|-1}$ has $-c_1$ as its coefficient.
	Using Combinatorial Nullstellensatz \cite{alon1999}, 
	there exists at least one point in $S_1 \times \cdots \times S_n$
	where $c_2F-c_1G$ is non-zero which is a contradiction.
	\qed
\end{proof}

	
%

	\subsubsection*{Proof of Theorem \ref{thm:unrestrictchap3}}
	
	\begin{statement}[Statement of Theorem \ref{thm:unrestrictchap3}]
		Let $n$ be a positive integer and $k \in [n]$. Then,
		$\gamma(n,[1,n-k],[2,n]) = n-1$, where $1 \leq k \leq \lceil\frac{n}{2}\rceil$.
	\end{statement}

	\begin{proof}
		
		From Statements \ref{p:eq2} and \ref{p:eq3} of Proposition \ref{prop:chap3easychap3}, we know that
		in order to prove Theorem \ref{thm:unrestrictchap3}, 
		we only need to   
		establish a lower bound of $n-1$ for $\gamma(n,[1,n-1],[2,n])$.
		
		Let $\mathcal{B}$ denote the set of all the $2^{n}-2$ non-monochromatic bicolorings of $[n]$. 
		Let $\mathcal{A}$ be a SUR of minimum cardinality for $\mathcal{B}$.
		Let $Y_B$ ($X_A$)  denote the $n$-dimensional $\pm 1$ vector (respectively, 0-1 vector) representing the bicoloring $B$
		(respectively, $A \subseteq [n]$)
		Consider the polynomial $P(Y_B)$, $B \in \mathcal{B}$.
		
		\begin{align}
		\label{eq:pf1}
		P(Y_B)= \prod_{A \in \mathcal{A}} \langle X_A,Y_B \rangle.
		\end{align}
		
		From the definition of $\mathcal{A}$, $P(Y_B)$ vanishes on all non-trivial bicolorings of $[n]$.
		Now, consider the following polynomial $P'(X)$.
		
		\begin{align}
		P'(X=(x_1,\ldots,x_n))= P(Y_B=(1-2x_1,\ldots,1-2x_n)) (x_1+\ldots+x_n - n).
		\end{align}
		
		$P'(X)$ vanishes at every  $X\in \{0,1\}^n$ except at the point $(0,\ldots,0)$ : $P$ vanishes 
		at every $X\in \{0,1\}^n$ except the two points $(0,\ldots,0)$ and $(1,\ldots,1)$ and
		$(x_1+\ldots+x_n - n)$ vanishes at $(1,\ldots,1)$.
		$P'(X)$ has degree at most $deg(P)+1$ (note that one can repeatedly replace $x_i^2$ with $x_i$ since $x_i \in \{0,1\}$).
		Using Lemma \ref{lemma:nullchap3} with each $S_i=\{0,1\}$, $1 \leq i \leq n$,
		it follows that $deg(P)+1 \geq deg(P') \geq n$.
		So, $|\mathcal{A}|=deg(P) \geq n-1$.
		\qed
	\end{proof}
	
%
	
	\begin{remark}
Lemma \ref{lemma:nullchap3} can also be used to obtain an alternative proof of induction base case of the 
Cayley-Bacharach  theorem by Riehl and Graham \cite{riehl2003} (see Appendix \ref{app:2chap3}). 
An alternative proof of the above lower bound can also be obtained using  the Cayley-Bacharach  
theorem by Riehl and Graham \cite{riehl2003}.
	\end{remark}

	Note that in Section \ref{subsec:1.1chap3}, the underlying set $\mathcal{B}$ of all the non-trivial bicolorings of $[n]$, 
	has cardinality $|\mathcal{B}|= 2^n-2$. In this case, Theorem \ref{thm:unrestrictchap3} establishes that $\gamma(n,[1,n-1],[2,n])=
	n-1=\Theta(\log |\mathcal{B}|)$. In the following section, 
	we match the $O(\log |\mathcal{B}|)$ upper bound for slightly restricted sets $\mathcal{B}$ of bicolorings. 
	
	\subsection{Relation to hitting sets for arbitrary collection of bicolorings} 
	\label{subsec:1.2chap3}
	
	Let $\mathcal{S}$ denote a collection of subsets of $[n]$.
	A subset $V \subseteq [n]$ is a \emph{hitting set} for $\mathcal{S}$
	if for every $S \in \mathcal{S}$, $V \cap S$ is non-empty.
	Let $H(\mathcal{S})$ denote a minimum cardinality hitting set of $\mathcal{S}$.
	The decision version of the Hitting set problem is:
	``Given the pair $(\mathcal{S},[n])$ and an integer $k$ as input, decide whether there exists a hitting set of cardinality
	at most $k$ for  $\mathcal{S}$''.

	
	\begin{lemma}
		Let $\mathcal{B}=\{B_0,\ldots,B_{m-1}\} \subseteq \{-1,+1\}^n$ be  a family of bicolorings
		of $[n]$.
		Construct the family $\mathcal{C}=\{C_1,\ldots,C_{2m}\}$
		where $C_{2i+1}=B_i(+1)$ and $C_{2i+2}=B_i(-1)$, for $0 \leq i \leq m-1$.
		Let $H=\{h_1,h_2,h_3,\ldots\}$ denote a hitting set for $\mathcal{C}$.
		Define $\mathcal{A} = \{(h_1,h_q)|h_q \in H, q > 1\}$.
		Then, $\mathcal{A}$ is a SUR for $\mathcal{B}$ of cardinality $|H|-1$.
		\label{claim:1chap3}
	\end{lemma}
	\begin{proof}
		For the sake of contradiction, assume that $B_i \in \mathcal{B}$ has no unbiased representative in $\mathcal{A}$.
		Assume that $h_1 \in B_i(+1)$.
		Since $H$ is a hitting set for $\mathcal{C}$, there exists some $h_q \in H$ such that $h_q$ hits $C_{2i+2}$ (and, thereby $B_i(-1)$).
		Then, the pair $(h_1,h_q)$ is an unbiased representative for $B_i$, a contradiction to our assumption.
		So, $h_1 \not\in B_i(+1)$. But this implies that $h_1 \in B_i(-1)$.
		A similar contradiction can be obtained in this case.
		\qed
	\end{proof}
	
	Let  $\mathcal{B}$ be restricted to a special family of bicolorings:
	the number of +1's for each $B \in \mathcal{B}$ lies in the range $\epsilon n$ and $(1- \epsilon)n$, 
	i.e., $\epsilon n \leq |B(+1)| \leq (1-\epsilon) n$, for some fixed $0 < \epsilon < \frac{1}{2}$.
	Construct the family $\mathcal{C}$ as above and let $d$ be the VC-dimension of $\mathcal{C}$.
	Note that every $C \in \mathcal{C}$ has size at least $\epsilon n$,
	for some fixed $\epsilon < \frac{1}{2}$.
	Using a result of Haussler and Welzl \cite{haussler1987} which was improved by Komlos et al. \cite{Komlos1992},
	we can get an `epsilon net' $H$ (which is a hitting set for $\mathcal{C}$)  of cardinality at most $\frac{d}{\epsilon}(\ln \frac{1}{\epsilon}+2\ln \ln \frac{1}{\epsilon}+6)$ (see Corollary 15.6 of \cite{pach1995combinatorial} for this exact bound).
	Using Lemma \ref{claim:1chap3}, it follows that we can construct a SUR for 
	$\mathcal{B}$ of cardinality $\frac{d}{\epsilon}(\ln \frac{1}{\epsilon}+2\ln \ln \frac{1}{\epsilon}+6)-1$.
	Since any family $\mathcal{C}$ of VC-dimension $d$ has cardinality at least $2^d$,
	this establishes an $O(\log |\mathcal{C}|) = O(\log |\mathcal{B}|)$  upper bound for  the cardinality of any optimal SUR
	under no restriction on set sizes. We state the result as a proposition below.
	
	\begin{proposition}
	Let $0 \leq \epsilon \leq \frac{1}{2}$ be a constant.	
	Let  $\mathcal{B}$ be a family of bicolorings, where $\epsilon n \leq |B(+1)| \leq (1-\epsilon) n$, for each $B \in \mathcal{B}$. Let $\mathcal{C}$ be the family constructed from $\mathcal{B}$ as in Lemma \ref{claim:1chap3}. 
	Let $d$ be the VC-dimension of $\mathcal{C}$. Then, we can construct a SUR for 
	$\mathcal{B}$ of cardinality $\frac{d}{\epsilon}(\ln \frac{1}{\epsilon}+2\ln \ln \frac{1}{\epsilon}+6)-1$.
	\end{proposition}

	In both Section \ref{subsec:1.1chap3} and \ref{subsec:1.2chap3}, the $O(\log |\mathcal{B}|)$ cardinality SURs contained sets of 
	small sizes (2-sized sets) as well.
	In what follows, we study the problem of SURs made of large cardinality sets.
	In order to obtain a similar $O(\log |\mathcal{B}|)$ bound for such a SUR, we inevitably introduce some error in the representation.
	
	%


	\subsection{Analysis with bias in representation}
	
	Consider the problem of estimation of $\gamma(\mathcal{B})$  for a set of bicolorings in terms of 
	$|\mathcal{B}|$, where (i) the number of +1's in each $B \in \mathcal{B}$ lies in the range $\{\alpha n, \alpha n+1, \ldots, (1-\alpha) n\}$ for some $0 < \alpha < \frac{1}{2}$, and (ii) each set in the SUR is of cardinality exactly $r$, for some $2 \leq r \leq n $.
	Choosing $r$ elements, namely $x_1,\ldots,x_r$, from $[n]$ independently and uniformly at random,
	the probability $p$ that a fixed bicoloring $B \in \mathcal{B}$ does not have $\left\langle Y_B,X_A\right\rangle=0$, where $A=\{x_1,\ldots,x_r\}$, is at most
	\[1-\binom{r}{\frac{r}{2}} \left(\frac{\alpha n}{n} \right)^{\frac{r}{2}} \left(\frac{(1-\alpha) n}{n} \right)^{\frac{r}{2}} \leq 1-C\frac{2^r}{\sqrt{r}}{\alpha}^{\frac{r}{2}}(1-\alpha)^{\frac{r}{2}} < e^{-C\frac{2^r}{\sqrt{r}}{\alpha}^{\frac{r}{2}}(1-\alpha)^{\frac{r}{2}}}, \text{where $C=\frac{1}{\sqrt{\pi}}$}.\]
	
	Let $\mathcal{A}$ be constructed by choosing $t$ $r$-element sets into $\mathcal{A}$ independently,
	where each $r$-element set is chosen as described above.
	Using union bound, 	the probability that some $B \in \mathcal{B}$ has $\left\langle Y_B,X_A\right\rangle\neq 0$ for all $A \in \mathcal{A}$, is $|\mathcal{B}|(e^{-C\frac{2^r}{\sqrt{r}}{\alpha}^{\frac{r}{2}}(1-\alpha)^{\frac{r}{2}}})^t$.
	This gives an upper bound of $\frac{\sqrt{r}}{C 2^r{\alpha}^{\frac{r}{2}}(1-\alpha)^{\frac{r}{2}}} \ln(|\mathcal{B}|)$ for $|\mathcal{A}|$. Using Proposition \ref{prop:nkr1chap3}, the case when $k=\frac{n}{2}$ and $r=2$ yields a asymptotically tight example for this upper bound.
	We have the following proposition.
	
	\begin{proposition}
		Let $\mathcal{B}$ denote a set of bicolorings, where  the number of +1's in each $B \in \mathcal{B}$ lie in the range $\{\alpha n, \alpha n+1, \ldots, (1-\alpha) n\}$ for some $0 < \alpha < \frac{1}{2}$. Let $\mathcal{A}$ denote a minimum cardinality SUR
		for $\mathcal{B}$, where each $A \in \mathcal{A}$ has cardinality exactly $r$. Then,
		\begin{align}
		|\mathcal{{A}}| \leq \frac{\sqrt{r}}{C2^r{\alpha}^{\frac{r}{2}}(1-\alpha)^{\frac{r}{2}}} \ln(|\mathcal{B}|),  
		\label{ineq:chap3:just}
		\end{align}
		where $C=\frac{1}{\sqrt{\pi}}.$
		\label{prop:chap3justchap3}
	\end{proposition}

When $\alpha=\frac{1}{2}-\epsilon$, for some $0 \leq \epsilon < \frac{1}{2}$, Inequality \ref{ineq:chap3:just} becomes
\begin{align}
|\mathcal{{A}}| \leq \frac{\sqrt{r}}{C(1-4 \epsilon^2)^{\frac{r}{2}}} \ln(|\mathcal{B}|).
\label{ineq:chap3:just1chap3}
\end{align}

Using the fact that $(1- \frac{1}{m+1})^m \geq \frac{1}{e}$,
the right hand term is at most $e^{(\frac{4 \epsilon^2}{1-4 \epsilon^2})\frac{r}{2}}\sqrt{\pi r}\ln |\mathcal{B}|$.
Therefore, when $r \in O(1)$, we have an $O(\ln |\mathcal{B}|)$ upper bound for any optimal SUR consisting of $r$ sized sets for $\mathcal{B}$. However, if $r$ is any increasing function in $n$, the upper bound given by Proposition \ref{prop:chap3justchap3} is large (even if $\epsilon = \frac{1}{n}$, the term $\frac{\sqrt{r}}{C(1-4 \epsilon^2)^{\frac{r}{2}}} \ln(|\mathcal{B}|)$ is $\Omega(\sqrt{r}\ln |\mathcal{B}|)$).
For large values of $r$, in order to obtain an $O(\ln(|\mathcal{B}|))$ upper bound for $|\mathcal{A}|$, one may allow some error in representation studied in the following section.
Let $\mathcal{B}$ denote the set of all bicolorings $B \in \{-1,+1\}^n$,
where $|B(+1)-B(-1)| \leq d$, for some $d \in \mathbb{N}$.
Our problem is to find a small sized family $\mathcal{A}$ for $\mathcal{B}$ such that
\begin{enumerate}
	\item each $A \in \mathcal{A}$ is reasonably large;
	\item for every $B \in \mathcal{B}$, there exists a set $A \in \mathcal{A}$ such that
	$|\left\langle Y_B,X_A\right\rangle| \leq \Delta$, where $\Delta=\Delta(r,d,n)$ is as small as possible.
\end{enumerate}
	
	\subsubsection*{Proof of Theorem \ref{thm:biasedchap3}}
	
	\begin{statement}[Statement of Theorem \ref{thm:biasedchap3}]
		Let $r' \in [r \pm \lceil\frac{r}{2}\rceil]$, where $r \geq 8$ is an integer.		
		Let $\mathcal{B}$ denote the set of all bicolorings $B \in \{-1,+1\}^n$,
		where $|B(+1)-B(-1)| \leq d$, for some $d \in \mathbb{N}$.
		Then, with high probability, one can construct a family $\mathcal{A}$ of cardinality at most $\ln |\mathcal{B}|$
		in $O(n|\mathcal{B}|\ln |\mathcal{B}|)$ time 
		consisting of $r'$-sized subsets such that for every $B \in \mathcal{B}$, there exists a set $A \in \mathcal{A}$ with
		$|\left\langle Y_B,X_A\right\rangle| \leq e\sqrt{r}+\frac{dr}{n}$.
	\end{statement}
	
	\begin{proof}
		
		We construct a set $A \subset [n]$ of size $r'\in [r \pm \lceil\frac{r}{2}\rceil]$ by picking each element of $[n]$ into $A$ independently with probability $\frac{r}{n}$.
		Let $X_A=(a_1,\ldots,a_n)$ denote the corresponding random vector
		where each $a_i \in \{0,1\}$.
		Note that $|A|=\sum_{i=1}^n a_i$.
		So, using linearity of expectation,
		$(\mu=)\mathbb{E}[|A|]=\sum_{i=1}^n \mathbb{E}[a_i]=r$.
		Moreover, since $a_i$'s are independent,
		$Var[|A|]=\sum_{i=1}^n Var[a_i]=r(1-\frac{r}{n})$.
		So, using the following form of Chernoff's bound 
		$P\left(|X-\mu|> \Delta \mu\right)  < (\frac{e^{\Delta}}{(1+\Delta)^{(1+\Delta)}})^{\mu} + (\frac{e^{-\Delta}}{(1-\Delta)^{(1-\Delta)}})^{\mu}$, we get,
		$P\left(|\sum_{i=1}^n a_i - r| > 0.5 r \right) < 0.72$, for $r \geq 8$.
		So, we can sample a family  $\mathcal{A}$ of cardinality $t$ ($t$ to be chosen later)
		consisting of sets of size $r'\in [r \pm \frac{r}{2}]$.
		
		Let $B \in \mathcal{B}$ be a bicoloring,
		where $B(+1)-B(-1) = d_1$, where $-d \leq d_1 \leq d$.
		Let $Y_B=(b_1,\ldots,b_n)$ denote the corresponding bit vector, where
		each $b_i \in \{-1,+1\}$.
		Let $Y=\left\langle Y_B,X_A\right\rangle$.
		Since $Y=\sum_{i=1}^n a_ib_i$,
		$Y$ becomes a random variable (note that $a_ib_i$ can take values $\{-1,0,1\}$ and are independent).
		So, $\mathbb{E}[Y]=\sum_{i=1}^n b_i\mathbb{E}[a_i]=\frac{d_1r}{n}$.
		It follows that $Var[Y]=\sum_{i=1}^n b_i^2 Var[a_i]=r(1-\frac{r}{n})$.
		So, using Chebyshev's inequality, we get,
		$P\left(|Y - \frac{d_1 r}{n}| \geq e \sqrt{r}\right) \leq \frac{1}{e^2}(1-\frac{r}{n}) < \frac{1}{e^2}$.
		That is, the probability that $|\langle Y_B, X_A\rangle| > \frac{d_1 r}{n}+ e \sqrt{r}$ is at most $\frac{1}{e^2}$.
		Let $E$ denote the bad event that some $B \in \mathcal{B}$ has $|\left\langle Y_B,X_A\right\rangle| > \frac{d r}{n}+e\sqrt{r}$ for all $A \in \mathcal{A}$.
		Using union bound, $P\left(E\right) \leq |\mathcal{B}|(\frac{1}{e^2})^t$.
		Setting $|\mathcal{B}|(\frac{1}{e^2})^t$ to  at most $\frac{1}{2}$,
		we get, $t \geq \ln |\mathcal{B}|$.
		
		Independently choose $100 t$ subsets of $[n]$ (call this collection $\mathcal{D}$), where each $D \in \mathcal{D}$ is constructed by picking an element of $[n]$ independently with probability $\frac{r}{n}$.
		Let $\mathcal{C} \subseteq \mathcal{D}$ be the sub-collection of $r'$-sized subsets in $\mathcal{D}$, where $r' \in \lceil\frac{r}{2}\rceil$.
		Then, $E[|\mathcal{C}|] \geq 28t$. Since $Var[|\mathcal{C}|] \leq 25t$, with high probability, $|\mathcal{C}| \geq 10t$.
		Partition $\mathcal{C}$ into $t$-sized sets.
		With high probability, one of the parts will form our desired family $\mathcal{A}$ that is a SUR (with restricted error)
		for $\mathcal{B}$.
		\qed
	\end{proof}
	
\emph{\noindent Comparison between Theorem \ref{thm:biasedchap3} and Proposition \ref{prop:chap3justchap3}:} 
Expressing $d$ in Theorem \ref{thm:biasedchap3} in terms of $\alpha$ in Proposition \ref{prop:chap3justchap3}, $(1-2\alpha)n=d$.
So, $\epsilon=\frac{1}{2}-\alpha=\frac{d}{2n}$. Substituting this value of $\epsilon$ in Inequality \ref{ineq:chap3:just1chap3},
we get a SUR of cardinality $\Omega(\sqrt{r}\ln |\mathcal{B}|)$ with no error for $\mathcal{B}$.


	\section{When cardinalities of sets  in the `SUR' and +1's in the bicolorings are restricted}
	\label{sec:2chap3}
	For any $k$-bicoloring $B$ of $[n]$, and any $A \subseteq [n]$,
	if $A$ is an unbiased representative for $B$, then $2 \leq |A| \leq 2k$:
	otherwise,  $\left\langle Y_B,X_A\right\rangle \neq 0$.
	Recall that $\gamma(n,k,r)= \gamma(\mathcal{B})$,
	where 
	(i) $\mathcal{B}$ is the collection of the  $\binom{[n]}{k}$ distinct $k$-bicolorings,
	(ii) $\gamma(\mathcal{B})$ is the cardinality of an optimal
	SUR $\mathcal{A}$ for $\mathcal{B}$, and,
	(iii) each $A \in \mathcal{A} $ has cardinality exactly $r$.
	We have the following propositions.
	
	\begin{proposition}
		$ \max(\lceil\frac{n-k}{r} \rceil, \lceil\frac{k}{r} \rceil) \leq \gamma(n,k,r)$.
		\label{prop:nkr1chap3}
	\end{proposition}
	
	\begin{proof}
		Consider the case when $k \leq \lfloor\frac{n}{2} \rfloor$. Given a SUR $\mathcal{A}$ of cardinality $\lfloor\frac{n-k}{r} \rfloor$ consisting of $r$-sized subsets, there exists a $k$-sized subset (say, $S$) of $[n]$ that is completely disjoint from the union of these $r$-sized subsets. The bicoloring with the points in $S$ colored +1 and the points in $[n] \setminus S$ colored -1 	does not have any unbiased representative in $\mathcal{A}$.
		\qed
	\end{proof}
	
	\begin{proposition}
		$ \frac{2}{r(r-1)} \gamma(n,k-1,r-2) \leq \gamma(n,k,r) \leq (n-r+1)\gamma(n,k-1,r-2)$, for $r \geq 4$.
		\label{prop:nkr2chap3}
	\end{proposition}
	See Appendix \ref{app:1} for a proof of Proposition \ref{prop:nkr2chap3}.
%

	A simple averaging argument gives the following lower bound. 
	\begin{align}
	\gamma(n,k,r) \geq  \frac{\binom{n}{k}} {\binom{r}{\frac{r}{2}} \binom{n-r}{k-\frac{r}{2}}}.
	\label{ineq:nkrlowchap3}
	\end{align}

	To establish an upper bound, 
	we reduce this problem to a covering problem and then make use of a result by Lov\'{a}sz and Stein \cite{STEIN1974391,LOVASZ1975383}.
	\begin{defn}
		Given a family $\mathcal{F}$ of subsets of some finite set $X$, the cover number $Cov(\mathcal{F})$ of $\mathcal{F}$
		is the minimum number of members of $\mathcal{F}$ whose union includes all the points in $X$.
	\end{defn}

	\begin{theorem}\cite{STEIN1974391,LOVASZ1975383,jukna2011extremal}
		If each member of $\mathcal{F}$ covers at most $a$ elements and each element in  $X$ is covered by at least $v$
		members of $\mathcal{F}$, then
		\[Cov(\mathcal{F}) \leq \frac{|\mathcal{F}|}{v}(1+\ln a).\]
		\label{thm:lov-steinchap3}
	\end{theorem}
	We have the following theorem.
	
\begin{theorem}
	\label{thm:nkr1chap3}
	Let $n$ be an integer, $r,k \in [n]$, $2 \leq r \leq 2k$ and $r$ is even. Then,
	\[\frac{\binom{n}{k}} {\binom{r}{\frac{r}{2}} \binom{n-r}{k-\frac{r}{2}}} \leq \gamma(n,k,r) \leq \frac{\binom{n}{k}}{{\binom{r}{\frac{r}{2}} \binom{n-r}{k-\frac{r}{2}}}} 
	\left(1+  0.7r + \ln (\binom{n-r}{k-\frac{r}{2}}) \right).\]
\end{theorem}
	
	\begin{proof}
		Consider the following construction of a uniform family of  subsets based on the $\binom{n}{[k]}$ distinct $k$-bicolorings
		and $\binom{n}{r}$ distinct $r$-sized subsets of $[n]$.
			\begin{construction} 
			Corresponding to each distinct $k$-bicoloring  $B$ in $\binom{[n]}{k}$, we add a point $v_B$ to $X$.
			Corresponding to each distinct $r$-sized subset  $A$ in $\binom{[n]}{r}$, we add a set $e_A$ to $\mathcal{F}$,
			where $e_A$ is the collections of all $v_B$'s such that $\left< X_A,Y_B \right>=0$. So, 
			$e_A$ `covers' $v_B$ if and only if $v_B \in e_A$.
			\label{const:1chap3}
		\end{construction}
		So,  $|X|=\binom{n}{k}$, $|\mathcal{F}|=\binom{n}{r}$.
		Clearly, $a={\binom{r}{\frac{r}{2}} \binom{n-r}{k-\frac{r}{2}}}$,
		$v=\binom{k}{\frac{r}{2}} \binom{n-k}{\frac{r}{2}}$.
		It follows from the construction that $\gamma(n,k,r) \leq Cov(\mathcal{F})$.
		So, from Theorem \ref{thm:lov-steinchap3}, we have
		
		\begin{align}
		\gamma(n,k,r) \leq \frac{\binom{n}{r}} { {\binom{k}{\frac{r}{2}} \binom{n-k}{\frac{r}{2}}} }
		\left(1+ \ln({\binom{r}{\frac{r}{2}} \binom{n-r}{k-\frac{r}{2}}}) \right).
		\label{ineq:lov-steinchap3}
		\end{align}
		
		%
		Double counting $(B,A)$ pairs, where $B$ is a $k$-bicoloring and $A$ is a $r$-sized subset that covers $B$, we get
		
		\begin{align}
		\binom{n}{k}\binom{k}{\frac{r}{2}} \binom{n-k}{\frac{r}{2}}=\binom{n}{r}{\binom{r}{\frac{r}{2}} \binom{n-r}{k-\frac{r}{2}}}.
		\label{ineq:lov-stein1chap3}
		\end{align}
		
		
		Combining Inequalities \ref{ineq:lov-steinchap3} and \ref{ineq:lov-stein1chap3}, and from Inequality \ref{ineq:nkrlowchap3}, 
		Theorem \ref{thm:nkr1chap3} follows.
		\qed
	\end{proof}

	Since Lov\'{a}sz-Stein method is deterministic and constructive, 
	the above reduction gives a deterministic polynomial time algorithm for obtaining a SUR.
	Moreover, from Theorem \ref{thm:nkr1chap3}, it follows that $\gamma(n,k,r)$ is $O(k\ln n)$ approximable 
	($k+0.2r+(k-\frac{r}{2})\ln(\frac{n-r}{k-\frac{r}{2}})$  to be precise) and when 
	$k=\frac{r}{2}$, the approximation factor becomes $O(r)$ ($1+0.7r$ to be exact).
	However, if $k \leq \log_4 \log_4 (n^{0.5-\epsilon})$ and $r=2k$, for some $0<\epsilon < 0.5$, then this upper bound can be improved further.
	
	\subsection{Tight upper bounds under restrictions}

	From Construction \ref{const:1chap3}, it is clear that the approximation factor for  $\gamma(n,k,r)$
	in Theorem \ref{thm:nkr1chap3} comes as a consequence of the approximation factor for  the cover number given by Lov\'{a}sz-Stein
	Theorem. So, tighter bounds for  the cover number should translate into tighter bounds for $\gamma(n,k,r)$.
	Let $v(B,D)$ denote the number of $r$-sized sets that are unbiased representatives for both $B$ and $D$, for any
	pair $(B,D)$ of $k$-bicolorings, where $B \neq D$.
	Let $v_{pair}=\displaystyle\max_{\substack{B,D\in \binom{[n]}{k},\\ B \neq D}} v(B,D)$.
	R\"{o}dl nibble method \cite{Rodlnibble1985,alon2004probabilistic} establishes asymptotically tight bounds for 
	the cover number provided the uniformity $a$ of the family $\mathcal{F}$ in Construction 
	\ref{const:1chap3} is fixed, $v \rightarrow \infty$, and $v_{pair} \in o(v)$.
	Alon et al. \cite{Alon2003} relaxed the condition to $a=o(\log v)$ provided $v_{pair} \in o(\frac{v}{e^{2a}\log v})$.
	In the estimation of $\gamma(n,k,r)$, if  $k \leq \log_4 \log_4 (n^{0.5-\epsilon})$ and $r=2k$, for any $0 < \epsilon < 0.5$, using  Construction \ref{const:1chap3}, it follows that $a < 2^r \in O(\log n)$ and $\log n \in o(\log v)$. 
	So, in order to 
	prove Theorem  \ref{thm:nkrtightchap3},
	it suffices to show that $v_{pair} \in o(\frac{v}{e^{2a}\log v})$.
	
	\begin{lemma}
	$v_{pair} \in o(\frac{v}{e^{2a}\log v})$, when $r=2k$ and $k \leq \log_4 \log_4 (n^{0.5-\epsilon})$, for any $0 < \epsilon < 0.5$.
	\label{lemma:nkrtightchap3}
	\end{lemma}
	
	\begin{proof}
	In order to prove the lemma, it is important to note that 
	$v(B,D)$ depends intrinsically on the cardinality of $B(+1) \cap D(+1)$.
	Let $S$ be some $r$-sized subset of $[n]$. Let $i_B=S \cap (B(+1) \setminus D(+1))$,
	$i_D=S \cap (D(+1) \setminus B(+1))$, $j_{BD}=S \cap (B(+1) \cap D(+1))$ and $j_{\overline{BD}}=S \cap ([n] \setminus (B(+1) \cup D(+1))$ (see Figure \ref{fig:1chap3}). 
	So, $S=i_B \cup i_D  \cup j_{BD} \cup j_{\overline{BD}}$.
	If $S$ is an unbiased representative for $B$, then  
	$|i_B|+|j_{BD}|=|i_D|+|j_{\overline{BD}}|=\frac{r}{2}$. If $S$ is an unbiased representative of $D$,
	then $|i_D|+|j_{BD}|=|i_B|+|j_{\overline{BD}}|=\frac{r}{2}$. 
	Therefore, if $S$ is an unbiased representative of both $B$ and $D$, then 
	(i) $|i_B|=|i_D|$ ($=i$, say),
	(ii) $|j_{BD}|=|j_{\overline{BD}}|$ ($=j$, say), and
	(iii) $2i+2j=r=2k$.
	Let $x=|B(+1) \cap D(+1)|$.
	We have,
	\begin{align}
	v(B,D) = \sum_{\substack{i,j:j \leq x,\\ i \leq k-x,\\ i+j = \frac{r}{2}}} \binom{x}{j} \binom{n-2k+x}{j} \left(\binom{k-x}{i} \right)^2.  
	\label{eq:tightnibblechap3}
	\end{align}
	
	\begin{figure}
		\centering
		\includegraphics*[scale=0.8]{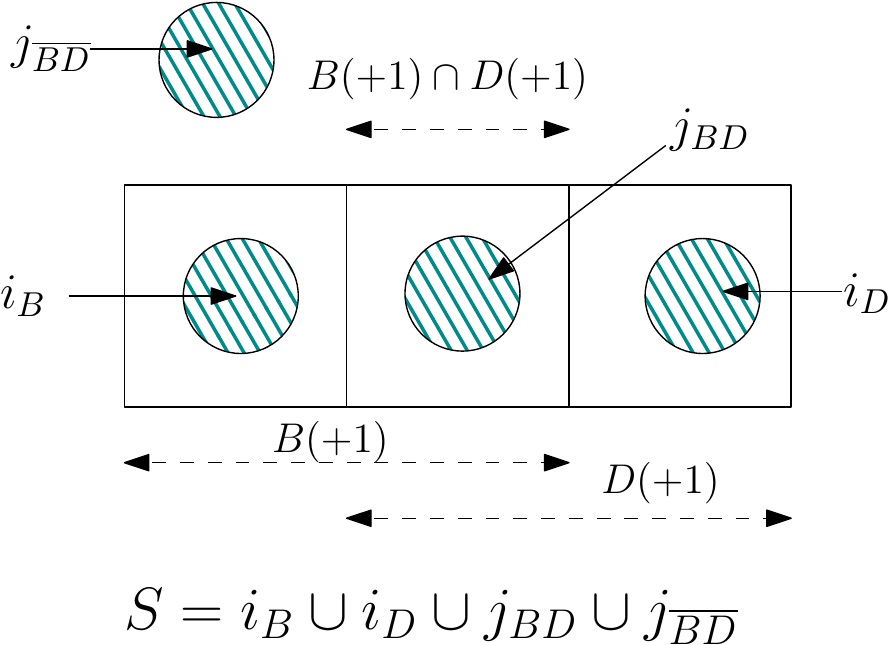}
		\caption{
			$S$ is some $r$-sized subset of $[n]$. Let $i_B=S \cap (B(+1) \setminus D(+1))$,
			$i_D=S \cap (D(+1) \setminus B(+1))$, $j_{BD}=S \cap (B(+1) \cap D(+1))$ and $j_{\overline{BD}}=S \cap ([n] \setminus (B(+1) \cup D(+1))$. So, $S=i_B \cup i_D  \cup j_{BD} \cup j_{\overline{BD}}$.
			If $S$ is an unbiased representative for $B$, then  
			$|i_B|+|j_{BD}|=|i_D|+|j_{\overline{BD}}|$. If $S$ is an unbiased representative of $D$,
			then $|i_D|+|j_{BD}|=|i_B|+|j_{\overline{BD}}|$. So, if $S$ is an unbiased representative of both $B$ and $D$, then $|i_B|=|i_D|$ and $|j_{BD}|=|j_{\overline{BD}}|$.}
		\label{fig:1chap3}
	\end{figure}
	
	Since $|B(+1)|=|D(+1)|=k$, applying Condition (iii), we get $x=j$  and $k-x=i$.
	In other words, if $S$ is an unbiased representative of cardinality $r=2k$ for both the $k$-bicolorings $B$ and $D$,
	$B \cup D \subset S$.
	So, for any pair $B,D$ of $k$-bicolorings, exactly one term in the summation of Equation \ref{eq:tightnibblechap3} remains valid, namely
	$\binom{x}{x} \binom{n-2k+x}{x} \left(\binom{k-x}{k-x} \right)^2$.
	For instance, when $x=k-1$, $v(B,D) = \binom{n-k-1}{k-1}$; when $x=k-2$, $v(B,D) = \binom{n-k-2}{k-2}$, etc.
	Therefore, $\frac{v(B,D)}{v(B',D')}= \Omega(\frac{n}{k})$ if $|B(+1) \cap D(+1)|=k-1$ and $|B'(+1) \cap D'(+1)| \leq k-2$.
	So, $v_{pair}=v(B,D)$, when $|B(+1) \cap D(+1)|=k-1$ provided $r=2k$. 
	Thus, $v_{pair} = \binom{n-k-1}{\frac{r}{2}-1}$, when $r=2k$.
	Computing $\frac{v_{pair}}{v}$, 
	\begin{align}
	\frac{v_{pair}}{v}=
	\frac{\binom{n-k-1}{\frac{r}{2}-1}}{\binom{k}{\frac{r}{2}} \binom{n-k}{\frac{r}{2}}}
	=\frac{r}{2(n-k)}.
	\end{align}	
	
	Note that $\log v=O(r \log n)$, $e^{2a} \leq n^{1-2\epsilon}$ since $k \leq \log_4 \log_4 (n^{0.5-\epsilon})$.
	So, $\frac{v_{pair}e^{2a}\log v}{v} = O(\frac{r^2 \log n}{n^{2\epsilon}}) \rightarrow 0$,
	when $n \rightarrow \infty$.
	\qed
	\end{proof}
	
	\subsubsection*{Proof of Theorem \ref{thm:nkrtightchap3}}
	
	\begin{statement}[Statement of Theorem \ref{thm:nkrtightchap3}] For sufficiently large values of $n$,
\[\frac{\binom{n}{k}} {\binom{2k}{k}} \leq \gamma(n,k,2k) \leq \frac{\binom{n}{k}} {\binom{2k}{k}}(1+o(1)), \]
	provided $k \leq \log_4 \log_4 (n^{0.5-\epsilon})$, for any $0 < \epsilon < 0.5$.
	\end{statement}
	\begin{proof}
From Lemma \ref{lemma:nkrtightchap3}, and using the result of Alon et al. \cite[Corollary 1.3]{Alon2003} to obtain coverings, 
the proof follows.
\qed
	\end{proof}

	
%

	\subsection{$\gamma(n,k,r)$ when $k=n/2$}
	
	Let $\mathcal{B}$ denote the set of all $\binom{n}{\frac{n}{2}}$ distinct $\frac{n}{2}$-bicolorings.
	It is not hard to see that $\mathcal{A}=\{\{1,2\},\{1,3\},\allowbreak \ldots, \{1,\frac{n}{2}+1\}\}$
	is a SUR of cardinality $\frac{n}{2}$ for $\mathcal{B}$.
	Together with Proposition \ref{prop:nkr1chap3}, this establishes 
	$\frac{n}{4} \leq \gamma(n,\frac{n}{2},2) \leq \frac{n}{2}$.
	It is easy to see that $\gamma(n,\frac{n}{2},n) = 1$.
	For arbitrary values of $r$, from Theorem \ref{thm:nkr1chap3} and Proposition \ref{prop:nkr1chap3}, we have,
	\begin{align}
\max\left(\left\lceil\frac{n}{2r}\right\rceil, c_1\sqrt{\frac{r(n-r)}{n}}\right) \leq \gamma(n,\frac{n}{2},r) \leq c_2n\sqrt{\frac{r(n-r)}{n}} \text{, where $c_1$ and $c_2$ are constants.}
\label{ineq:nkrremotechap3}
	\end{align}
	When $r=\frac{n}{2}$, this establishes a lower bound and upper bound of  $\Omega(\sqrt{n})$ and  $O(n\sqrt{n})$,
	respectively. In general, when $r=f(n)$ is an increasing function in $n$, this establishes sub-linear lower bounds for 
	$\gamma(n,\frac{n}{2},r)$.
	
	%
	%

	We use an extension of a theorem of Frankl and R\"{o}dl \cite{frankl1987}
	given by Keevash and Long \cite{keevash2017frankl} to obtain a linear lower bound on $\gamma(n,k,r)$ under certain restrictions
	on $k$ and $r$. Let $\mathcal{D} \subseteq [q]^n$
	be a $q$-ary code. For any $x,y \in \mathcal{D}$, the Hamming distance between
	$x$ and $y$ is the number of indices where $x(i)\neq y(i)$, for $1 \leq i \leq n$.
	The code $\mathcal{D}$ is called \textit{$d$-avoiding} if
	the Hamming distance between no pair of code-words in $\mathcal{D}$ is $d$. The following upper bound for {$d$-avoiding} codes is given in \cite{keevash2017frankl}.
	
	\begin{theorem}\cite{keevash2017frankl}
		Let $\mathcal{D} \subseteq [q]^n$ and let $\epsilon$ satisfy $0 < \epsilon < \frac{1}{2}$.
		Suppose that $\epsilon n < d < (1-\epsilon)n$ and $d$ is even if $q =2$. If $\mathcal{D}$ is $d$-avoiding,
		then $|\mathcal{D}| \leq q^{(1-\delta)n}$, for some positive constant $\delta=\delta(\epsilon)$.
		\label{thm:keeveshchap3}
	\end{theorem}
	
We have the following lower bound for $\gamma(n,k,r)$, when $r=2c$ 
for any odd integer $c \in \{1,\ldots,\frac{n}{2}\}$ and $\epsilon n < k < (1-\epsilon)n$, for some $0 < \epsilon < 0.5$.

\subsubsection*{Proof of Theorem \ref{thm:keevashchap3}}
\begin{statement}[Statement of Theorem \ref{thm:keevashchap3}]
	Let $r=2c$
	for any odd integer $c \in \{1,\ldots,\frac{n}{2}\}$. 
	Let $k$ be an even integer, where $\epsilon n < k < (1-\epsilon)n$ for some $0 < \epsilon < 0.5$.
	Then, $\gamma(n,k,r) \geq \delta n$, 	
	where $\delta=\delta(\epsilon)$ is some real positive constant.
\end{statement}

\begin{proof}
	Let $\mathcal{B}=\{B_1,\ldots,B_{\binom{n}{k}}\}$ denote the set of all the bicolorings of $[n]$
	consisting of exactly $k$ +1's.
	We construct a family  $\mathcal{C}=\{C_1,\ldots,C_{\binom{n}{k}}\}$, where $C_i$
	corresponds to the +1 colored points of $B_i \in \mathcal{B}$.
	Let $\mathcal{A}$ be a SUR for $\mathcal{B}$, where each $A \in \mathcal{A}$ has cardinality exactly $2c$ for some odd number $c \in [n]$.
	Note that $\left\langle Y_{B_i},X_A\right\rangle=0$ implies
	that $\left\langle X_{C_i},X_A\right\rangle=c$, where
	$X_{C_i}$ denotes the 0-1 incidence vector corresponding to the set $C_i$.
	Let $V \subset \{0,1\}^n$ denote the vector space spanned by the vectors $X_A$'s,
	$A \in \mathcal{A}$, over $\mathbb{F}_2$.
	Let $V^{\perp} \subset \{0,1\}^n$ denote the subspace orthogonal to $V$.
	Since $\mathcal{A}$ is a SUR for $\mathcal{B}$, it follows that for every $C_i$, there exists
	a set $A \in \mathcal{A}$ such that $\left\langle X_{C_i},X_A\right\rangle=1 (\mod 2)$ (since $c$ is odd).
	Therefore, $X_{C_i} \not\in V^{\perp}$, for all $X_{C_i} \in \mathcal{C}=\binom{[n]}{k}$.
	In other words, $V^{\perp}$ does not contain any vector consisting of exactly $k$ ones.
	Moreover, observe that for any $x,y \in V^{\perp} $,
	the number of ones in $x+y$ is same as the Hamming distance between $x$ and $y$.
	Thus, $V^{\perp}$ is $k$-avoiding.
	Since $\epsilon n < k < (1-\epsilon) n$ and $k$ is even, from Theorem \ref{thm:keeveshchap3}, it follows that there exists a positive constant $\delta=\delta(\epsilon)$ such that $|V^{\perp}| \leq 2^{n(1-\delta)}$.
	So, dimension of $V^{\perp}$ is at most $n(1-\delta)$.
	Therefore, it follows that dimension of $V$ is at least $\delta n$.
	\qed
\end{proof}

\begin{corollary}
	$\gamma(n,\frac{n}{2},r) \geq \delta n$ provided $\frac{n}{2}$ is even and $\frac{r}{2}$ is odd,
	for some $0<\delta<1$. 
	\label{cor:keevashchap3}
\end{corollary}

Let $\frac{n}{2}$ be even and $\frac{r}{2}$ be odd. From Inequality \ref{ineq:nkrremotechap3}, we have $\gamma(n,\frac{n}{2},r) \in O(n\sqrt{r})$. When $r$ is a constant, using Corollary \ref{cor:keevashchap3}, this upper bound is asymptotically tight.
%
However, for larger values of  $r$, there can be a large gap (up to $O(\sqrt{n})$ when $r \in \Omega(n)$) between the upper and the lower bound. In what follows, we address the problem for a special case when $r=\frac{n}{2}$ and establish a better upper bound of $\frac{n}{2}$  on $\gamma(n,\frac{n}{2},\frac{n}{2})$.

\begin{lemma}
	$\gamma(n,\frac{n}{2},\frac{n}{2}) \leq \frac{n}{2}$, where $\frac{n}{2}$ is any even integer.
	\label{lemma:nn2upchap3}
\end{lemma}

\begin{proof}
	Let $\mathcal{B}$ denote the set of all the bicolorings with equal number of +1's and -1's.
	Let $A_1=\{1,2,\ldots,\frac{n}{2}\},
	A_2=\{2,3,\ldots,\frac{n}{2}+1\},\ldots,
	A_{\frac{n}{2}}=\{\frac{n}{2},\frac{n}{2}+1,\ldots,n-1\}$.
	Let $c_i(B)=\left\langle Y_B,X_{A_{i}}\right\rangle$.
	For any $B \in \mathcal{B}$, it is not hard to see that
	each $c_i(B)$ is even and $|c_i(B) - c_{i+1}(B)|  \in \{0,2\}$.
	Since the bicolorings consist of equal number of +1's and -1's,
	$c_{\frac{n}{2}}(B) \leq -c_1(B)+2$ if $c_1(B) \geq 0$,
	and $c_{\frac{n}{2}}(B) \geq -c_1(B)-2$ if $c_1(B) < 0$.
	In particular, we have $c_1(B)c_{\frac{n}{2}}(B) \leq 0$.
	Since $|c_i(B) - c_{i+1}(B)|  \in \{0,2\}$, this implies the existence of an index $i$ such that 
	$c_i(B)=\left\langle Y_B,X_{A_{i}} \right\rangle=0$.
	This concludes the proof that $\gamma(n,\frac{n}{2},\frac{n}{2}) \leq \frac{n}{2}$.
	\qed
\end{proof}

From Corollary \ref{cor:keevashchap3} and Lemma \ref{lemma:nn2upchap3}, we have the following theorem.

\begin{theorem}
	$\gamma(n,\frac{n}{2},\frac{n}{2}) \leq \frac{n}{2}$.
	Moreover, $\gamma(n,\frac{n}{2},\frac{n}{2}) \geq \delta n$
	if $n/2$ is even and $n/4$ is odd, for some $0<\delta<1$. 
\end{theorem}

		\section{Inapproximability of the SUR problem}
		\label{sec:3chap3}
		Firstly, we establish a hardness result of the hitting set problem for a special family of subsets.
	
	\begin{defn}
		A family $\mathcal{F}$ of subsets of $[n]$ is \emph{complement closed on $[n]$} if for all $F \in \mathcal{F}$,
		$[n] \setminus F \in \mathcal{F}.$
	\end{defn}

	\begin{proposition}
		Let $n$ be an integer, $n \geq 4$. 
		No deterministic polynomial time algorithm can 
		approximate the hitting set problem for complement closed families on $[n]$
		to within a factor of $(1-\Omega(1))\frac{\ln n}{2.34}$ of the optimal, unless P=NP.
		\label{prop:specialchap3}
	\end{proposition}
	
	\begin{proof}
		For the sake of contradiction, assume that there exists an algorithm $ALG$ that approximates
		the hitting set for complement closed families on $[n]$ to within a factor of $(1-\Omega(1))\frac{\ln n}{2.34}$ of the optimal.
		We obtain a contradiction to this assumption by the following reduction from the general hitting set problem.
		
		Given a pair $(\mathcal{S}',[n])$ as input to the general hitting set problem, we extend the universe to $[n+1]$ by adding the element $n+1$. We construct $\mathcal{S}$ as follows: $\mathcal{S}=\mathcal{S}' \cup \{[n+1]\setminus S|S \in \mathcal{S}'\}$.
		Let $OPT(\mathcal{S})$ ($OPT(\mathcal{S}')$) denote an optimal solution to the hitting set problem on $\mathcal{S}$ (respectively, $\mathcal{S}'$). Let $ALG(\mathcal{S})$ denote a hitting set outputted by $ALG$ on $\mathcal{S}$ as input.
		
		Observe that 
		\begin{align}
		|OPT(\mathcal{S}')| \leq |OPT(\mathcal{S})|  \leq |OPT(\mathcal{S}')|+1 \leq  2|OPT(\mathcal{S}')|.
		\label{eq:inapprox1chap3}
		\end{align}
		From our assumption, 
		we know that $|OPT(\mathcal{S})|  \leq |ALG(\mathcal{S})| \leq (1-\Omega(1))\frac{\ln (n+1)}{2.34} |OPT(\mathcal{S})|
		<(1-\Omega(1))\frac{\ln n}{2}|OPT(\mathcal{S})|$, for $n \geq 4$.
		Note that $ALG(\mathcal{S})$ is a valid hitting set for $\mathcal{S}'$.
		So, $|OPT(\mathcal{S}')| \leq |OPT(\mathcal{S})|\leq |ALG(\mathcal{S})| \leq (1-\Omega(1))\frac{\ln n}{2} |OPT(\mathcal{S})|
		< (1-\Omega(1))\frac{\ln n}{2} 2 \cdot |OPT(\mathcal{S}')|= (1-\Omega(1))\ln n \allowbreak |OPT(\mathcal{S}')|$.
		Therefore, $ALG$ is a $(1-\Omega(1))\ln n$ factor approximation algorithm for the general hitting set problem.
		However, Dinur and Steurer \cite{dinur2014} proved that it is impossible to approximate the set cover problem
		to a factor of $(1-\Omega(1))\ln n$ of the optimal, unless P=NP.
		\qed
	\end{proof}
	

	We use Proposition \ref{prop:specialchap3} to establish the following hardness result
	for the system of unbiased representative problem.
	
\subsubsection*{Proof of Theorem \ref{thm:inappchap3}}

\begin{statement}[Statement of Theorem \ref{thm:inappchap3}]
Let $r \leq  (1-\Omega(1))\frac{\ln n}{2.34}$, where $n \geq 4$ is an integer.
Then, no deterministic polynomial time algorithm can 
approximate the system of unbiased representative problem for a family of bicolorings on $[n]$
to within a factor $(1-\Omega(1))\frac{\ln n}{2.34r}$ of the optimal 
when each set chosen in the representative family is required to have its cardinality at most $r$, unless P=NP.
\end{statement}
	
	\begin{proof}
		We prove Theorem \ref{thm:inappchap3} by a reduction from an instance of the hitting set problem on complement closed familes.
		Let $\mathcal{S}$ be a complement closed family on $[n]$.
		From $\mathcal{S}$, we construct a family $\mathcal{B}$ of bicolorings on $[n]$ in the following way:
		$\mathcal{B}=\{B| B(+1)=S, B(-1)=[n]\setminus S, S \in \mathcal{S}\}$.
		For the sake of contradiction, assume that there exists an algorithm $ALG$ that 
		approximates the system of unbiased representative problem for any family of bicolorings on $[n]$ 
		to within a factor $f$ of the optimal, where $1 \leq f \leq (1-\Omega(1))\frac{\ln n}{2.34r}$
		and each set in the SUR
		is required to have its cardinality at most $r$.
		Let $OPT_{\text{HIT}}(\mathcal{S})$ ($OPT_{\text{SUR}}(\mathcal{B})$) denote an optimal solution to the hitting set problem
		(respectively, the system of unbiased representative problem) on $\mathcal{S}$ (respectively, $\mathcal{B}$). Let $ALG(\mathcal{B})$ denote a SUR outputted by $ALG$ with $\mathcal{B}$ as its input.
		%
		%
		Then, executing $ALG$ on $\mathcal{B}$ as input, we obtain a SUR $\mathcal{A}$ for $\mathcal{B}$ such that
		(i) $2 \leq |A| \leq r$ for each $A \in \mathcal{A}$,
		(ii) $|ALG(\mathcal{B})|=|\mathcal{A}| \leq f\cdot |OPT_{\text{SUR}}(\mathcal{B})|$,
		for some $1 \leq f \leq (1-\Omega(1))\frac{\ln n}{2.34r}$.
		Let $V=\cup_{A\in \mathcal{A}} A$.
		It follows that $|V| \leq r |\mathcal{A}|$ and $V$ is a hitting set for $\mathcal{S}$.
		
		From Lemma \ref{claim:1chap3}, we know that $|OPT_{\text{SUR}}(\mathcal{B})| \leq |OPT_{\text{HIT}}(\mathcal{S})|-1$. Therefore,
		\[|OPT_{\text{HIT}}(\mathcal{S})| \leq |V| \leq r \cdot |ALG(\mathcal{B})| \leq r\cdot f\cdot |OPT_{\text{SUR}}(\mathcal{B})| < r\cdot f\cdot |OPT_{\text{HIT}}(\mathcal{S})|.\]
		So, $ALG$ is a $(r\cdot f)$-factor approximation algorithm for computing hitting set of $\mathcal{S}$.
		Since $1 \leq f \leq (1-\Omega(1))\frac{\ln n}{2.34r}$, this is a contradiction to Proposition \ref{prop:specialchap3}.
		\qed
	\end{proof}
	
	\begin{remark}
		Consider the case when the family $\mathcal{B}$ is restricted to a special family of bicolorings, where
		the number of +1's (or -1's) for each $B \in \mathcal{B}$ is exactly one, i.e. $|B(+1)|=1$ (or $|B(-1)|=1$).
		Then, the problem of system of unbiased representatives reduces to an edge cover problem \cite{white1971minimum,murty19821} on a complete graph $G$,
		where for each $B \in \mathcal{B}$, a vertex $v_{B(+1)}$ (respectively, $v_{B(-1)}$) is added to $V(G)$.
		So, this reduction makes the SUR problem polynomial time solvable for such families of bicolorings.
	\end{remark}

	\section*{Acknowledgment}
	
The authors thank Prof. Niloy Ganguly for helpful discussions on the problem.

	%
	
	{\small
		\bibliography{refer}}
	\bibliographystyle{plain}

	\appendix
%
	
%
%
	\section{Proof of induction base case of Theorem \ref{thm:cayleychap3}}\label{app:2chap3}
	\begin{theorem}\cite{riehl2003}
				Given the $n$ quadratics in $n$ variables $x_1(x_1-1),\ldots,x_n(x_n-1)$ with $2^n$
				common zeros, the maximum number of those common zeros a polynomial $P$ of
				degree $k$ can go through without going through them all is $2^n-2^{n-k}$.
				\label{thm:cayleychap3}
	\end{theorem}
	\begin{proof}
		The proof is by induction on $n$.
		When $k=0$, we have nothing to prove.
		So, we consider all the degree $k$ polynomials $P$
		on $k+1$ variables as the base case.
		For the sake of contradiction, assume that
		$P$ is a polynomial of degree $k$ on $k+1$ variable and 	
		it misses only one common zero of $x_1(1-x_1),\ldots,x_{k+1}(1-x_{k+1})$.
		Then, using Lemma \ref{lemma:nullchap3},
		it follows that degree of $P$ must be $k+1$, which is a contradiction.
		This completes the proof of the induction base case.
		
		The rest of the proof is exactly same as given in \cite{riehl2003}.
		\qed
	\end{proof}
	
	\section{Proof of Proposition \ref{prop:nkr2chap3}} \label{app:1}
	\begin{statement}[Statement of Proposition \ref{prop:nkr2chap3}]
				$ \frac{2}{r(r-1)} \gamma(n,k-1,r-2) \leq \gamma(n,k,r) \leq (n-r+1)\gamma(n,k-1,r-2)$, for $r \geq 4$.
	\end{statement}
	\begin{proof}
		Let $\mathcal{B}_i$ denote the set of all the bicolorings consisting of exactly $i$ +1's, for $i\in \{k,k-1\}$.
		Let $\mathcal{A}_{r-2}$ denote a family of $(r-2)$-sized subsets that is an optimal unbiased representative
		family for $\mathcal{B}_{k-1}$.
		For any $A \in \mathcal{A}_{r-2}$, let $\bar{A}=[n]\setminus A= \{x_1,\ldots,x_{n-r+2}\}$. 
		For each $A \in \mathcal{A}_{r-2}$, we construct $(n-r+1)$ $r$-sized subsets as follows:
		$A^1=A \cup \{x_1,x_2\}$, $A^2=A \cup \{x_1,x_3\}$, $\cdots$, 	$A^{n-r+1}=A \cup \{x_1,x_{n-r+2}\}$.
		Let $\mathcal{A}_r= \cup_{A \in \mathcal{A}_{r-2}} \{A^1, \cdots, A^{n-r+1}\}$.
		To see that $\mathcal{A}_r$ is a system of unbiased representative
		for $\mathcal{B}_{k}$, consider any $B \in \mathcal{B}_{k}$ and a $(k-1)$-sized subset $B' \subset B_k$.
		Let $A' \in \mathcal{A}_{r-2}$ has $\left \langle Y_{B'}, X_{A'}\right \rangle=0$.
		From the construction, it follows that there is at least one  $A \in \{A'^1, \cdots, A'^{n-r+1}\}$ such that
		$\left \langle Y_{B}, X_{A}\right \rangle=0$. 
		
		For the lower bound,  consider a SUR $\mathcal{A}$ for $\mathcal{B}_{k}$ of size $\gamma(n,k,r)$.
		For each $A \in \mathcal{A}$,
		let $\mathcal{F}_A$ denote the family of $\binom{r}{r-2}$ distinct $(r-2)$-sized subsets of $A$.
		Then, $\mathcal{A}'=\cup_{A \in \mathcal{A}}F_A$ is 
		an  unbiased representative
		family for $\mathcal{B}_{k-1}$ where each set in the family is of size exactly $(r-2)$.
		\qed
	\end{proof}

\end{document}